\documentclass{article}

\usepackage{amsmath}
\usepackage{graphicx}
\usepackage{amssymb}
\usepackage{amsfonts}
\usepackage{latexsym}
\usepackage{amscd}
\usepackage[mathscr]{euscript}
\usepackage{xy} \xyoption{all}
\usepackage[utf8]{inputenc} 

\usepackage{tikz}
\usetikzlibrary{shapes,arrows}
\usepackage[utf8]{inputenc}
\usepackage{color}

\usepackage{setspace}

\newtheorem{theorem}{Theorem}[section]

\newtheorem{corollary}[theorem]{Corollary}

\newtheorem{example}[theorem]{Example}

\newtheorem{lemma}[theorem]{Lemma}

\newtheorem{proposition}[theorem]{Proposition}
\newtheorem{remark}[theorem]{Remark}

\newenvironment{proof}[1][Proof]{\noindent\textbf{#1.} }{\ \rule{0.5em}{0.5em}}

\begin{document}

\title{Finite Dimensional Representations of Leavitt Path Algebras}
\author{Ayten Ko\c{c} $^{*}\> $, $\>$  Murad  \"{O}zaydın$^{**}$\\ \\}



\maketitle

\begin{abstract}
When $\Gamma$ is a row-finite digraph we classify all finite dimensional modules 
of the Leavitt path algebra $L(\Gamma)$ via an explicit Morita equivalence given by an effective combinatorial (reduction) algorithm on the digraph  $\Gamma$. The category of (unital) $L(\Gamma)$-modules is equivalent to a full subcategory of quiver representations of $\Gamma$. However the category of finite dimensional representations of $L(\Gamma)$ is tame in contrast to the finite dimensional quiver representations of $\Gamma$ which are almost always wild.
\end{abstract}


\section{Introduction}

Our purpose is to classify all finite dimensional representations of the Leavitt path algebra $L(\Gamma)$ of a row-finite di(rected )graph $\Gamma$. The main tools we employ are that the category of $L(\Gamma)$-modules is equivalent to a full subcategory of quiver representations of $\Gamma$ satisfying a natural isomorphism condition (Theorem \ref{teorem}) and an explicit Morita equivalence defined by a graph theoretical (reduction) algorithm (Theorem \ref{Morita}).\\

From the quiver representation viewpoint a module $M$ is a functor from the small category $\Gamma$ (whose objects are the vertices in $\Gamma$ and whose morphisms are the paths in $\Gamma$) to the category of $\mathbb{F}$-vector spaces. This functor is an $L(\Gamma)$-module if and only if it satisfies the isomorphism condition (I) of Theorem \ref{teorem}. The reduction algorithm helps to reduce the complexity of the source category $\Gamma$, in particular the cycles relevant for finite dimensional indecomposable modules become loops (a single arrow starting and ending at the same vertex).\\

That the Leavitt path algebra $L(\Gamma)$ of a finite $\Gamma$ has a nonzero finite dimensional representation is equivalent to an easy graph theoretical condition: $\Gamma$ has a \textit{maximal sink} or a \textit{maximal cycle}, that is, a sink or a cycle such that there is no path from any other cycle to it. For Leavitt path algebras this property is properly squeezed between two important algebraic notions: it is implied by the cycles of $\Gamma$ being mutually disjoint (which is equivalent to $L(\Gamma)$ having polynomial growth \cite[Theorem 5]{aajz12} and also to
all simple $L(\Gamma)$-modules being finitely presented \cite[Theorem 4.5]{ar14}). In turn, it implies that $L(\Gamma)$ has IBN (Invariant Basis Number), but it is equivalent to neither \cite[Corollary 6.5]{ko1}.\\

Let's describe the classification of finite dimensional $L(\Gamma)$-modules for  a finite $\Gamma$ here (rather than a row-finite $\Gamma$) to avoid the annoying technicality of \textit{finitely many predecessors}.
Finite dimensional simple $L(\Gamma)$-modules are of two types: (i) projective modules corresponding to maximal sinks; (ii) the others that are in one-to-one correspondence with pairs $(C,f(x))$ where $C$ is a maximal cycle in $\Gamma$ and $f(x) \in \mathbb{F}[x]$ is an irreducible polynomial with $f(0)=1$.
When $F$ is algebraically closed, type (ii) is parametrized by a disjoint union of finitely many copies of $\mathbb{F}^\times:=\mathbb{F}\setminus \{0\}$,
one for each maximal cycle in $\Gamma$. \\

Finite dimensional nonsimple indecomposable modules correspond to pairs $(C, f(x)^n)$ where $n> 1$, $C$ and $f(x)$ are as above. Hence, if $M$ and $N$ are finite dimensional simples then $Ext(M,N)=0$ unless $M$ and $N$ are isomorphic and of type (ii). Given a maximal sink or a pair $(C,f(x)^n)$, the corresponding indecomposable module and the associated quiver representation can be described explicitly. Once the indecomposables are classified, all finite dimensional representations are classified by Krull-Schmidt. \\ 

We refer to the excellent survey \cite{abr15} for the history and development of Leavitt path algebras. The precise definition of the Leavitt path algebra $L(\Gamma)$ of a digraph $\Gamma$ will be given in section 2 below where we also collect other definitions and facts that will be needed. In particular the concepts of the support subgraph, maximal sinks, maximal cycles and the theorem stating the equivalence of the category of unital $L(\Gamma)$-modules with a full subcategory of quiver representations of $\Gamma$ \cite[Theorem 3.2]{ko1} are relevant for our classification. \\

Section 3 is about a graph theoretic process we call the \textit{reduction algorithm} (defined on a row-finite digraph $\Gamma$) and its consequences. Reduction leaves the nonstable $K$-theory (i.e., the monoid $\mathcal{V}(L(\Gamma))$ of isomorphism classes of finitely generated projective $L(\Gamma)$-modules under direct sum) invariant. This enables us to show that dimension functions on a finite digraph form a free commutative monoid on the maximal sinks and cycles of $\Gamma$ (Theorem \ref{27}). The monoid of dimension functions also turn out to be isomorphic to the monoid homomorphisms from $\mathcal{V}(L(\Gamma))$ to the natural numbers  $\mathbb{N}$ under addition.\\

Section 4 starts with Theorem \ref{Morita} showing that the Leavitt path algebras of a digraph and any reduction of it are Morita equivalent (which is given explicitly in terms of the corresponding quiver representations). This generalizes the fact that \textit{source elimination} is Morita invariant \cite[Proposition 10]{abr15}. We classify all finite dimensional representations by first reducing the problem to the (finite) support subgraph, then applying the reduction algorithm and finally pulling back to the Leavitt path algebra via the Morita equivalence of Theorem \ref{Morita}.\\

The category of finite dimensional $L(\Gamma)$-modules, for a row-finite digraph $\Gamma$, is equivalent to the direct sum of the categories of finite dimensional vector spaces indexed by the maximal sinks with finitely many predecessors and the categories of finite dimensional modules over the algebra of Laurent polynomials indexed by the maximal cycles with finitely many predecessors (Theorem \ref{notasyon}). In particular the category of finite dimensional representations of the Leavitt path algebra $L(\Gamma)$ of a finite digraph $\Gamma$ is tame in contrast to the finite dimensional quiver representations of $\Gamma$ which are wild unless the connected components of the undirected $\Gamma$ are Dynkin diagrams of types $A, \> D, \> E, \> \tilde{A}, \> \tilde{D}, \> \tilde{E}$ by a celebrated theorem of Gabriel \cite{dw05}.\\
 
 \section{Preliminaries}



\medskip
\subsection{Leavitt Path Algebras}
\medskip
A \textbf{di}(rected )\textbf{graph} $\Gamma$ is a four-tuple $(V,E,s,t)$ where $V$ is the set of vertices, $E$ is the set of arrows (directed edges), $s$ and $t:E \longrightarrow V$ are the source and  the target functions. The digraph $\Gamma$ is \textit{finite} if $E$ and $V$ are both finite. $\Gamma$ is \textit{row-finite} if $s^{-1}(v)$ is finite for all $v$ in $V$. Given $V' \subseteq V$ the \textit{induced subgraph} on $V'$ is $\Gamma':= (V',E',s',t')$ with $E':=  s^{-1}(V') \cap t^{-1}(V')$ ; $\> s':= s\vert_{E'}$ ; $\> t':= t \vert_{E'}$. A subgraph is \textit{full} if it is  the induced subgraph on its vertices.
 
 \begin{remark}
 A digraph is also called an "oriented graph" in graph theory, a "diagram" in topology and category theory,  a "quiver" in representation theory, usually just a "graph" in $C^*$-algebras and Leavitt path algebras. The notation above for a digraph is standard in graph theory. However $Q=(Q^0,Q^1, s,t)$ is more common in quiver representations while $E=(E^0, E^1,s,r)$ is mostly used in graph $C^*$-algebras and in Leavitt path algebras. We prefer the graph theory notation which involves two more letters but no superscripts. As in quiver representations we view $\Gamma$ as a small category, so "arrow" is preferable to "edge", similarly for "target" versus "range".
 \end{remark}
 
A vertex $v$ in $V$ is a \textit{sink} if $s^{-1}(v)= \emptyset$; it is a \textit{source} 
if $t^{-1}(v)= \emptyset$. An \textit{isolated vertex} is both a source and a sink. If $t(e)=s(e)$ then $e$ is  a \textit{loop}. A \textbf{pseudo-source} is a vertex $v$ where $t^{-1}(v)$ consists of a single loop. A \textit{path} of length $n>0$ is a sequence $p =e_{1}\ldots e_{n}$ such
that $t(e_{i})=s(e_{i+1})$ for $i=1,\ldots ,n-1.$ The source of $p$ is $s(p
):=s(e_{1})$  and the target of $p$ is $t(p ):=t(e_{n})$. A path $p$ of length 0 consists of a single vertex $v$ where $s(p) :=v $ and $t(p) := v$. We will denote the length of $p$ by $l(p)$.  A path $C=e_1e_2 \cdots e_n$ with $n>0$ is a \textit{cycle} if $s(C )=t(C )$ and $s(e_{i})\neq s(e_{j})$ for $i\neq j$. An arrow $e \in E$ is an \textit{exit} of the cycle $C=e_1e_2 \cdots e_n$ if there is an $i$ such that $ s(e)=s(e_i)$ but $e\neq e_i$. The digraph $\Gamma$ is \textit{acyclic} if it has no cycles. \\

There is a preorder defined on the set of sinks and cycles in $\Gamma$: we say that a cycle $C$ \textbf{connects to} a sink $w$  denoted by $C \leadsto w$ if there is a path from $C$ to $w$. Similarly $C\leadsto D$  if there is a path from the cycle C to the cycle D. This is a partial order if and only if the cycles in $\Gamma$ are mutually disjoint. A cycle is \textbf{minimal} with respect to $\leadsto$ if and only if it has no exit (sinks are always minimal). A cycle $C$ is \textbf{maximal} if no other cycle connects to $C$ (in particular, a maximal cycle is disjoint from all other cycles). A sink $w$ is maximal if there is no cycle $C$ which connects to $w$. \\

The \textbf{graph monoid $S(\Gamma)$} is the (additive) commutative monoid generated by $V$ subject to the relations: $v =\sum_{e \in s^{-1}(v)} te$ for all $ v \in V $ with $ 0< \vert s^{-1}(v) \vert < \infty$.\\


Given a digraph $\Gamma,$ the \textit{extended digraph} of $\Gamma$ is $\tilde{\Gamma} := (V,E \sqcup E^*, s~,t~)$ where $E^* :=\{e^*~|~ e\in E \}$,  the functions $s$ and $t$ are  extended as $s(e^{\ast}):=t(e)$ and $t(e^{\ast }):=s(e) $ for all $e \in E$. Thus the dual arrow $e^*$ has the opposite orientation of $e$. We want to extend $*$ to an operator defined on all paths of $\tilde{\Gamma}$: Let $v^*:=v$ for all $v$ in $V$, $(e^*)^*:=e $ for all $e$ in $E$ and $p^*:= e_n^* \ldots e_1^*$ for a path $p=e_1 \ldots e_n$ with $e_1, \ldots , e_n $ in $E \sqcup E^* $. In particular $*$ is an involution, i.e., $**=id$. \\


The \textbf{ Leavitt path algebra} of a digraph $\Gamma$ with coefficients in the field $\mathbb{F}$, as defined in \cite{aa05} and \cite{amp07}, is the $\mathbb{F}$-algebra $L_{\mathbb{F}}(\Gamma)$ generated by $V \sqcup E \sqcup E^*$ satisfying:\\
\indent(V)  $\quad \quad vw ~=~ \delta_{v,w}v$ for all $v, w \in V ,$ \\
\indent($E$)  $\quad \quad s ( e ) e  = e=e t( e)  $ for all $e \in E\sqcup E^*$, \\
\indent(CK1) $\quad e^*f ~ = ~ \delta_{e,f} ~t(e)$ for all $e,f\in E$, \\
\indent(CK2) $\quad v~  = ~ \sum_{s(e)=v} ee^*$  for all $v$ with $0< \vert s^{-1}(v) \vert < \infty $.\\


We will usually suppress the subscript $\mathbb{F}$ when we denote our algebras. If the digraph $\Gamma$ is fixed and clear from the context we may abbreviate $L(\Gamma)$ to $L$. From now on we will omit the parentheses to reduce notational clutter when the source and target functions $s$, $t$ are applied. \\

The relations (V) simply state that the vertices are mutually orthogonal idempotents. The relations (E) say that $e \in seLte$ and $e^* \in teLse$ for every $e \in E$. The Leavitt path algebra $L$ as a vector space is $\bigoplus vLw$ where the sum is over all pairs $(v,w) \in V \times V$ since $V\sqcup E\sqcup E^*$ generates $L$. If we only impose the relations (V) and ($E$) then we obtain $\mathbb{F}\tilde{\Gamma}$, the \textbf{path }(or \textbf{quiver}) \textbf{algebra} of the extended digraph $\tilde{\Gamma}$ : The paths in $\tilde{\Gamma}$ form a vector space basis of $\mathbb{F} \tilde{\Gamma}$, the product $pq$ of two paths $p$ and $q$ is their concatenation if $tp=sq$ and 0 otherwise. \\

We get the\textit{ Cohn path algebra} $C(\Gamma)$ when we impose the relations (CK1) in addition to (V) and ($E$). Hence $L(\Gamma)$ is a quotient of $C(\Gamma)$, which is a quotient of $ \mathbb{F}  \tilde{\Gamma}$. The abbreviation CK stands for Cuntz-Krieger. \\

For any arrow $e$ in $E$ we have $e^*e=te$ by (CK1). Consequently $p^*p=tp$ for any path $p$ of $\Gamma$. Hence for any two paths $p$ and $q$ of $\Gamma$ if $q=pr$ then $p^*q=p^*pr=r$, if $p=qr$ then $p^*q=(q^*p)^*=r^*$. As $e^*f=0$ when $e\neq f$ by (CK1), $p^*q=0$ unless the path $q$ is an initial segment of the path $p$  (i.e., $p=qr$) or $p$ is an initial segment of $q$ (i.e., $q=pr$). Thus the Cohn path algebra $C(\Gamma)$ and the Leavitt path algebra $L(\Gamma)$ are spanned by $\left\{pq^*\right\}$ where $p$ and $q$ are paths of $\Gamma$ with $tp=tq$. In fact this is a basis for $C(\Gamma)$
which can be shown by defining an epimorphism from $C(\Gamma)$ to a reduced semigroup algebra with this basis (we will not need this fact). In $L(\Gamma)$ however $\{pq^* : tp=tq \}$ is linearly dependent because of (CK2).\\

The algebras $ \mathbb{F}\Gamma$, $ \mathbb{F}\tilde{\Gamma}$, $C(\Gamma)$ and $L(\Gamma)$ are unital if and only if $V$ is finite, in which case the sum of all the vertices is the unit: It is clear that $\sum_{v \in V} v=1$ when $V$ is finite. For the converse, a given element in any these algebras is a finite linear combination of paths in $\tilde{\Gamma}$ and we can pick $v \in V$ which is not the source of any of these paths if $V$ is infinite. Now left multiplication by $v$ gives zero, so there is no unit in any of these algebras since $\mathbb{F}\Gamma$ is a subalgebra of $L(\Gamma)$ by \cite[Lemma 1.6]{goo09} and thus $v\neq 0$ in $L(\Gamma)$ for every $v \in V$, hence also in $C(\Gamma)$.\\ 

We will need the following easy and well-known fact: 
\begin{lemma} \label{ez}
If $v=sC$ where $C$ is a cycle  without exits in $\Gamma$ then $vL(\Gamma)v \cong \mathbb{F}[x,x^{-1}]$ via $x \leftrightarrow C^*$.\end{lemma}
\begin{proof}
Since $C$ has no exit $ee^*=se$ by (CK2) for any arrow $e$ on $C$. So $CC^*=v$ and also $C^*C=v$ by (CK1). Since $v=1$ in $vL(\Gamma)v$ we get $C^{-1}=C^*$ and a homomorphism from $\mathbb{F}[x,x^{-1}]$ to $vL(\Gamma)v$ is defined by sending $x$ to $C^*$. This homomorphism is onto: $vL(\Gamma)v$ is spanned by $\{ pq^* \> \vert \> sp=v=sq , \> tp=tq\}$. If $tp\neq v$ then $tp=ee^*$ by (CK2) where $e$ is the unique arrow with $se=tp$. So $pq^*=pe(qe)^*$ and we can repeat as needed to get $pq^*=C^mC^{*n}=C^{m-n}$ for some $m$ and $n$ in $\mathbb{N}$.
This homomorphism is also one-to-one because $\{C^n\}_{n \geq 1}$ is linearly independent in the path algebra $\mathbb{F}\Gamma$, hence also in $vL(\Gamma)v \subseteq L(\Gamma)$. 
\end{proof}\\

There is a $\mathbb{Z}$-grading on $ \mathbb{F}  \tilde{\Gamma}$ given by $deg(v)=0$ for $v$ in $V$, $deg(e)=1$ and $deg(e^*)=-1$ for $e$ in $E$. This defines a grading on all our algebras since the relations are all homogeneous. The linear extension of $*$ on paths induces a grade-reversing involutive anti-automorphism (i.e., $deg(\alpha^*)=-deg(\alpha)$ and $(\alpha \beta)^*=\beta^* \alpha^*$). Hence these algebras are $\mathbb{Z}$-graded $*$-algebras and the (graded) categories of left modules and right modules are equivalent for any of these algebras.\\



A subset $H$ of $V$ is \textit{hereditary} if for any path $p$, $sp \in H$ implies that $tp \in H$; $H$ is \textit{saturated} if $\{ te : se=v \} \subseteq H$ implies that $v \in H$, for every  $v \in V$ with $0<\vert s^{-1}(v) \vert < \infty $ \cite{aa05}. If $I$ is an ideal of $L(\Gamma)$ and $p$ is a path in $\Gamma$ with $sp \in I$ then $tp=p^*p=p^*(sp)p \in I$, also if $\{ te : se=v \} \subseteq I$ then $v=\sum_{se=v}  ee^*=\sum_{se=v}  e(te)e^* \in I$, so $I\cap V$ is hereditary and saturated. We have a Galois connection between the subsets of $V$ and the ideals of $L(\Gamma)$ given by $S \mapsto (S)$ and $I \mapsto I\cap V$ which gives a bijection between hereditary saturated subsets of $V$ and graded ideals of $L(\Gamma)$ when $\Gamma$ is a row-finite digraph \cite[Theorem 5.3]{amp07}.

\medskip

\subsection{$L(\Gamma)$-Modules and Quiver Representations}
\medskip
As each $v$ in $V$ is an idempotent, $vL$ is a cyclic projective $L$-module. We have  
a homomorphism $\phi_v: \bigoplus_{se=v} teL\longrightarrow vL$ sending $(\alpha_e)$ to $ \sum e\alpha_e$ for all $v$ with $0< \vert s^{-1}(v) \vert < \infty $. The relations (CK1) and (CK2) imply that this homomorphism is an isomorphism whose inverse sends $\beta$ in $vL$ to $(e^*\beta)$ in $\bigoplus teL$. \\

In fact, $L(\Gamma)$ can be defined as the Cohn localization of the path algebra $\mathbb{F}\Gamma$, without mentioning $E^*$, (CK1) or (CK2), making the homomorphisms analogous to the $\phi_v$ above  from $\bigoplus te\mathbb{F}\Gamma$ to $v\mathbb{F}\Gamma$ invertible \cite[Corollary 4.2]{ab10}, \cite[Proposition 3.3]{ko1}.\\

The isomorphisms $\phi_v$ also enable us to define a monoid homomorphism from the graph monoid $S(\Gamma)$ to $\mathcal{V} (L)$, the nonstable K-Theory of $L$, by sending $[v]$ to $[vL]$. A deep and important result in the subject is: 

\begin{theorem} (\cite[Theorem 3.5]{amp07}) The monoid homomorphism from $S(\Gamma)$ to $\mathcal{V} (L(\Gamma))$ defined  above is an isomorphism.
\end{theorem}

We will work in the category $\mathfrak{M}_{L }$ of unital (right) modules over $L:=L_{\mathbb{F}} (\Gamma)$. However $L$ has a 1 if and only if the vertex set $V$ is finite. Even if $V$ is infinite, we define a unital $L$-module as a module $M$ with the property that $ML=M$, i.e., for any $m$ in $M$ we can find $\lambda_1, \lambda_2, \ldots , \lambda_n$ in $L$ and $m_1, m_2, \ldots, m_n$ in $M$ so that $m=m_1\lambda_1+m_2 \lambda_2 +\cdots + m_n\lambda_n$. This condition is equivalent to the standard definition of unital (when $L$ has a 1) since $m1=(m_1\lambda_1+m_2 \lambda_2 +\cdots + m_n\lambda_n)1=m_1\lambda_1 1 +m_2 \lambda_2 1+\cdots + m_n\lambda_n1=m_1\lambda_1+m_2 \lambda_2 +\cdots + m_n\lambda_n=m$. The projective modules $vL$ for all $v \in V$ are unital. The category of unital modules is an abelian category with sums since it is closed under taking quotients, submodules, extensions, (arbitrary) sums (but not infinite products: if $V$ is infinite then the $L$-module $L^V$ is not unital).\\

For  any $M$ in $\mathfrak{M}_L$ and any $v$ in $V$ the linear map sending $f \in Hom_L (vL, M)$ to $f(v) \in  Mv$ gives an isomorphism. Applying the cofunctor $Hom_L( \>\underline{ \> \> \> } \>, M)$ to $\phi_v$ we get an isomorphism $Mv \longrightarrow \bigoplus Mte$. Thus, assigning the vector space $Mv$ to each vertex $v$ and the linear transformation from $Mse$ to $Mte$ given by right multiplication by $e$ for all $e \in E$ defines a quiver representation of $\Gamma$ satisfying the isomorphism conditions imposed by the $\{\phi_v \}$.  \\

 
 Recall that a quiver representation is a functor from the small category $\Gamma$ with objects $V$ and morphisms given by paths in $\Gamma$. The linear transformation assigned to the path $p$ from $Msp$ to $Mtp$ is given by right multiplication by $p$. Conversely, any quiver representation of $\Gamma$ satisfying the isomorphism conditions above yields an  $M $  in $\mathfrak{M}_L$: 

\begin{theorem} \label{teorem} \cite[Theorem 3.2]{ko1}
If $\Gamma= (V,E,s,t)$ is a row-finite digraph then the category $\mathfrak{M}_{L}$
 is equivalent to the full subcategory of quiver representations $\rho$ of $\Gamma$ satisfying the following condition $(I)$:\\
  $$\textit{For every nonsink } v\in V, \> \> \> \>  (\rho(e))_{se=v}: \rho ( v) \longrightarrow \bigoplus\limits_{se=v } \rho(te) \textit{ is an isomorphism.}$$ The module $M$ corresponding to the quiver representation $\rho$ is $M=\bigoplus_{v\in V} \rho(v)$.
\end{theorem}

 With the quiver representation viewpoint there is no need to  mention the generators $\{e^* : e \in E \}$ explicitly, they are implicit in the condition (I). Below we will frequently define modules by constructing the corresponding quiver representations.\\

A \textbf{dimension function} on a digraph $\Gamma$ is a function $d: V  \longrightarrow \mathbb{N}$  satisfying: $d(v) =\sum_{se=v} d(te)$ for all $ v \in V $ with $ 0< \vert s^{-1}(v) \vert < \infty$. Hence dimension functions on $\Gamma$ correspond exactly to monoid homomorphisms from the graph monoid $S(\Gamma)\cong \mathcal{V} (L)$ to the additive monoid of natural numbers $\mathbb{N}$.\\

$M \in \mathfrak{M}_L$ is of \textit{finite type} if $dim^{\mathbb{F}} (Mv)< \infty$ for all $v \in V$. Then $d(v):=dim^{\mathbb{F}}(Mv)$ is a dimension function by condition (I). When $V$ is finite, finite type is the same as finite dimensional since $M =\oplus_{v \in V}  Mv$ by Theorem \ref{teorem}. Theorem \ref{teorem} also implies that any dimension function can in fact be realized by an $L$-module \cite[Corollary 3.7]{ko1}.
\medskip

\subsection{Support Subgraph}
\medskip
For any $M \in \mathfrak{M}_L$ the \textbf{support subgraph} of $M$, denoted by $\Gamma_{_M}$, is the induced subgraph of $\Gamma$ on $V_{_M}:=\{v \in V \mid Mv \neq 0 \}$. It's easy to check that $V\setminus V_{_M}= \{ v \in V \>  \vert \> Mv=0 \}$  is hereditary and saturated, hence the ideal $I_{_M}$ generated by $V\setminus V_{_M}$ is the largest graded ideal of $L(\Gamma)$ contained in $AnnM$. Conversely,  $V\setminus H$  for any hereditary saturated subset $H$ of $V$ is the support subgraph of $L(\Gamma) / (H)$ regarded as an $L(\Gamma)$-module. Also $L(\Gamma_{_M})\cong L(\Gamma)/I_{_M}$ and we may regard $M$ as an $L(\Gamma_{_M})$-module whose $L(\Gamma)$-module structure is induced by the projection from $L(\Gamma)$ to $L(\Gamma_{_M})$. As an $L(\Gamma_M)$-module, $M$ has full support, that is, $Mv \neq 0$ for all $v \in V_M$ \cite[Section 5]{ko1}.

\begin{lemma} \label{exits} \cite[Lemma 6.1]{ko1} 
If  $M\in \mathfrak{M}_{L(\Gamma)}$ is of finite type then the cycles of its support subgraph $\Gamma_{_M}$ have no exits. 
\end{lemma}

\begin{proof}
If $v_1,..., v_n, v_{n+1}=v_1$ are consecutive vertices in a cycle of $\Gamma_{_M}$ then $dim(Mv_1) \geq dim(Mv_2) \geq \cdots \geq dim(Mv_n)\geq dim( Mv_1)$ by Theorem \ref{teorem}. Hence $dim(Mv_k) =dim(Mv_{k+1})$ for $k=1, \cdots, n$. It follows from Thoerem \ref{teorem} again that $Mte=0$ for $e\in s^{-1}(v_k)$
unless $te=v_{k+1}$. Thus the cycles of $\Gamma_{_M}$ have no exits. 
\end{proof}\\

 For any two vertices $v$ and $w$ in $V$ we write $v \leadsto w$ if there is a path $p$ in $\Gamma$ from $v$ to $w$. This defines a preorder (a reflexive and transitive relation). If $v$ and $w$ are on a cycle then $v \leadsto w$ and $w \leadsto v$. Let $U$ be the set of sinks and cycles of $\Gamma$. There is an induced preorder on $U$, also denoted by $\leadsto$. (This is a partial order on $U$ if and only if the cycles of $\Gamma$ are disjoint, equivalently when $L(\Gamma)$ has polynomial growth \cite[Theorem 5]{aajz12}.) A sink or a cycle $u \in U$ is \textbf{maximal} if $u' \leadsto u$ only if $u' =u$. \\

 The \textbf{predecessors} of $v$ in $V$ is $V_{\leadsto v} := \{ w \in V \mid w \leadsto v\} $. If $u$ and $w$ are two vertices on a cycle $C$ then they have the same predecessors, so $V_{\leadsto C} $ is well-defined. The \textbf{ predecessors subgraph}  $\Gamma_{\leadsto v}$ is the induced subgraph on $V_{\leadsto v}$, similarly $\Gamma_{\leadsto C} := \Gamma_{\leadsto v}$ where $v$ is any vertex on $C$.

\section{The Reduction Algorithm}
This section is about the consequences of a geometric (graph theoretic)  process we call the \textbf{reduction algorithm} defined on a row-finite digraph $\Gamma=(V,E)$: For a loopless nonsink $v \in V$, we replace each path $fg$ of length 2 such that $tf=v=sg$ with an arrow labeled $fg$ from $sf$ to $tg$ and delete $v$ and all arrows touching $v$. (Note that $fg$ denotes a path in $\Gamma$, but an arrow in its reduction.) In particular, if $v$ is a source but not a sink, then we delete $v$ and all arrows starting at $v$ without adding any new arrows. We may repeat this  as long as there is a loopless non-sink. Any digraph obtained during this process is called a \textbf{reduction} of $\Gamma$. 
If $\Gamma$ is finite then after finitely many steps we will reach a \textbf{complete reduction} of $\Gamma$,  which has no loopless nonsinks. A digraph in which every vertex is either a sink or has a loop is called \textbf{reduced}.\\

In the example below, $\Gamma_1$ and $\Gamma_2$ are reductions of the digraph $\Gamma_0$. The number of arrows from one vertex to another is indicated by the number above the arrow (so there are 3 arrows from $w$ to $x$). There are two reduction steps going from $\Gamma_1$ to $\Gamma_2$ and $\Gamma_2$ is a complete  reduction of $\Gamma_0$.

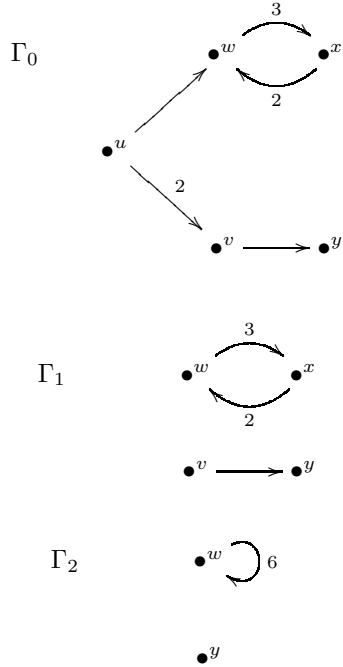
\begin{example}

$$
\Gamma_0 \quad \quad  \xymatrix{ & {\bullet}^{w} \ar@/^1pc/[r]^{3}  & {\bullet}^{x} \ar@/^1pc/[l]^{2} \\ {\bullet}^{u} \ar@{->}[ur] \ar@{->}[dr]^{2} & \\
                 & \bullet^{v} \ar@{->}[r]  & \bullet^{y}  }$$

\bigskip

 $$\Gamma_1 \quad \xymatrix{ & {\bullet}^{w}  \ar@/^1pc/[r]^{3}  & {\bullet}^{x} \ar@/^1pc/[l]^{2} \\ 
                 & \bullet^{v} \ar@{->}[r]  & \bullet^{y}  }$$

\bigskip

$$ \Gamma_2 \quad \xymatrix{ & {\bullet}^{w} \ar@(ur,dr)^{6} &\\ 
                 &  \bullet^{y}  }$$

\end{example}

\bigskip

\begin{remark}
The reduction algorithm can be described in terms of  the "six graph moves" originating in symbolic dynamics \cite{alps11}, \cite[Appendix 3]{abr15}: When the vertex $v$ to be eliminated is not a source, we perform an "in-split" where $t^{-1}(v)$ is partitioned so that each part consists of a single arrow. Then we perform  a "contraction" on each one of these arrows. Under the additional hypothesis of $L(\Gamma)$ being purely infinite simple all six graph moves yield Morita equivalences \cite{alps11}, \cite[Proposition 10]{abr15}. In Theorem \ref{Morita} below we prove that the reduction algorithm yields a Morita equivalence without any hypothesis on $L(\Gamma)$. \\

When $v$ is a source, reduction is exactly the same as "source elimination" which is known to give a Morita equivalence \cite[Proposition 10]{abr15}. Source elimination can also be obtained as a composition of an in-split partitioning all arrows $s^{-1}(v)$ in the singletons followed by contractions of all these arrows. 

\end{remark}

If $\Gamma'$ is a one step reduction eliminating the vertex $v$ of $\Gamma$ then there is a monomorphism $ \iota$ from $L(\Gamma')$ to $L(\Gamma)$ defined as:  $\iota(w)=w$ for all $w \in V'$, $\iota (e)=e $ and $\iota (e^*)= e^*$ if $e \in E' \cap E$, finally $\iota (fg)=fg \> , \> \iota( (fg)^*)=g^*f^*$ if $fg$ is a new arrow. (Again the input $fg$ is an arrow of $\Gamma'$ while the output $fg$ is a path of length 2 in $\Gamma$.) This defines a homomorphism since the relations are satisfied. In fact $\iota$ is a graded homomorphism where $L(\Gamma')$ has the standard $\mathbb{Z}$-grading ($deg(w)=0 $ for all $w \in V' \> ,\>  deg(e)=1 \> , \> deg(e^*)=-1$ for all $e \in E'$) but the grading in $L(\Gamma)$ is defined as $deg(w)=0 $ for all $w \in V$, $deg(e)=1, deg(e^*)=-1$ if $se\neq v$ and $deg(g)=deg(g^*)=0$ if $sg=v$. A graded ideal is generated by the vertices it contains (as a consequence of the one-to-one correspondence between graded ideals and hereditary, saturated subsets of vertices). Hence the graded ideal $Ker (\iota)$ is trivial since it contains no vertices, thus $\iota$ is one-to-one. \\

We may identify $L(\Gamma')$ with $Im(\iota)=L':=\{ \alpha \in L(\Gamma) \mid  v\alpha=0=\alpha v \} \> $:  Clearly $Im(\iota) \subseteq L'$. Conversely, $L'$ is spanned by elements of the form $pq^*$ where $p$ and $q$ are paths in $\Gamma$ with $tp=tq$ and $sp\neq v \neq sq $. If $tp=v$ then we can use $v=\sum_{g \in s^{-1}(v)} gg^*$ to express $pq^*=pvq^*$ as an element of $Im(\iota)$ thus $L'=Im(\iota)$.\\

If $\Gamma'$ is a reduction of $\Gamma$ then $\Gamma$ and $\Gamma'$ have the same set of sinks. A cycle may get shorter under reduction but it can not disappear. However, the number of cycles may increase as illustrated in the following example where the number of cycles increases from 2 to 3:

$$\Gamma :  
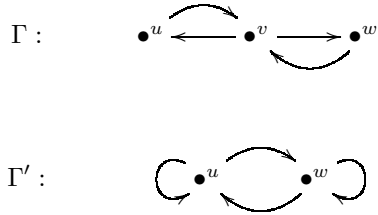
\xymatrix{& {\bullet}^{u} \ar@/^1pc/[r] &\ar [l]  {\bullet}^{v}  \ar [r]  & {\bullet}^{w} \ar@/^1pc/[l] }$$\\

$$\Gamma' :  \qquad  \xymatrix{ & {\bullet}^{u} \ar@(ul,dl)\ar@/^1pc/[r] & {\bullet}^{w} \ar@/^1pc/[l] \ar@(ur,dr)} \qquad $$\\

The digraph $\Gamma$ above also shows that complete reductions are not unique. If we chose to eliminate the vertices $u$ and $w$ (instead of $v$ as above) we obtain the rose with 2 petals which is not isomorphic to the $\Gamma'$ above. \\

However, if the cycles of a finite digraph $\Gamma$ are disjoint (that is, the Gelfand-Kirillov dimension of $L(\Gamma)$ is finite \cite[Theorem 5]{aajz12}) then the number of cycles does not change under reduction since each eliminated vertex effects at most one cycle. In fact, all complete reductions of $\Gamma$ are isomorphic when the cycles of $\Gamma$ are disjoint.


\begin{proposition} 
If $\Gamma$ is a finite digraph then $L(\Gamma)$ is finite dimensional if and only if any complete reduction of $\Gamma$ consists only of isolated vertices. In this case $L(\Gamma)$ is isomorphic to a direct sum of matrix algebras and the number of summands equals the number of (isolated) vertices of a complete reduction of $\Gamma$.
\end{proposition}

\begin{proof} 
$L(\Gamma)$ is  finite dimensional if and only if $\Gamma $ is finite and has no directed cycles, moreover $L(\Gamma)$ is isomorphic to a direct sum of matrix algebras with as many summands as the number of sinks of $\Gamma$  \cite[Corollaries 3.6 and 3.7]{aas07}. A closed path in a reduction of 
$\Gamma$ comes from a closed path in $\Gamma$. Hence any complete reduction of a finite acyclic $\Gamma$ will only have isolated vertices, corresponding to the sinks of $\Gamma$. Conversely, if any complete reduction of $\Gamma$ consists of isolated vertices then $\Gamma$ could not have had any cycles because a reduction may shorten a cycle, but it can not get rid of it.
\end{proof} \\

Reduction preserves the monoid of $\Gamma$ enabling us to classify all dimension functions.

\begin{lemma}
		If $\Gamma^{'}$ is a reduction of the digraph $\Gamma$, then
		 $S(\Gamma) \cong S(\Gamma^{'}) $.
	\end{lemma}	
	
	\begin {proof}	
	 It suffices to show that $S(\Gamma) \cong S(\Gamma^{'}) $ when $\Gamma^{'}$ is a one step reduction of 
		$\Gamma$. If $v$ is the deleted loopless nonsink, then $V^{'} =V \setminus \left\{v \right\} $ is the vertex set of $\Gamma^{'}$. Let $\Phi :S(\Gamma^{'}) \longrightarrow S(\Gamma) $ be induced by the inclusion of $V^{'}$ into $V$. In $S(\Gamma^{'})$
				$$ u=\sum_{se=u , ~ te \neq v} t(e) + \sum_{se=u, ~te=v} \left( \sum_{sf=v} tf \right)  \quad  \forall \> u \in V'$$
		by the definition of $\Gamma^{'}$. (Note that $sf=v$ implies $tf \neq v$ since $v$ is loopless.) Since, in $S(\Gamma)$ we have  $ u=\sum_{se=u} te$ and $v=\sum_{sf=v} tf$, $\Phi$ is a well-defined semigroup homomorphism. Now, $\Psi (u) =u $ if $u\neq v$ and $\Psi (v) =\sum_{sf=v} tf$ similarly defines $\Psi: S(\Gamma) \longrightarrow S(\Gamma^{'})$, which is the inverse of $\Phi$. 
				\end{proof}
\bigskip





 A sink $w$ becomes an isolated vertex in a complete reduction of finite $\Gamma$ if and only if $w$ is maximal (there are no cycles in $\Gamma$ from which there is a path to $w$). A cycle is maximal in $\Gamma$ if and only if it becomes a loop at a pseudo-source in a complete reduction. \\

\begin{lemma} \label{nonzero}
If $\Gamma$ is a reduced digraph and $ d:V \rightarrow \mathbb{N}$ is a dimension function then vertices with $d(v)>0 $ are either isolated vertices or  pseudo-sources.

\end{lemma}

\begin{proof} 
Let $T$ be the set of isolated vertices and pseudo-sources of $\Gamma$. If $v \notin T$ then there are at least two loops at $v$ or there is an arrow $e$ such that $te=v$. \\

If there are two loops $f, g$ at $v$ then $d(v) = \sum_{s(e)=v} d(te) \geq 2d(v)$. 
Hence we get $d(v)=0$. \\

If there is an arrow $e$ such that $te=v$ then $se$ is not a sink. Hence there exists a loop $h$ at $se$ since $\Gamma$ is a reduced digraph. Then $d(v)=0$  because $d(se)\geq d(se)+ d(te)= d(se)+d(v)$. 
\end{proof}

\begin{lemma} \label{psource}
Let $\Gamma$ be a reduced digraph. A dimension function is obtained by assigning arbitrary values to the isolated vertices and the pseudo-sources of $\Gamma$. Hence the monoid of homomorphisms from $S(\Gamma )$ to $\mathbb{N}$ is isomorphic to the  free commutative monoid generated by the isolated vertices and the pseudo-sources of $\Gamma$.


\end{lemma}

\begin{proof} 
We will check that any assignment of arbitrary natural numbers to pseudo-sources and isolated vertices will define a dimension function (where $d(v):=0$ for all the remaining vertices). There are no relations to check for sinks (including isolated vertices). If $v$ is a pseudo-source then there is a unique loop at $v$ and for any other arrow $e$ with $se=v$, the vertex $te$ is neither a pseudo-source nor an isolated vertex. Thus $d(te)=0$ by definition. Hence the relation at $v$ is $d(v)=d(v)+0$, which holds for any choice of $d(v)$. \\

If $v$ is neither a sink nor a pseudo-source then $d(v)=0$ by definition and for any $e$ with $se=v$ its target $te$ is neither isolated nor a pseudo-source. Hence the relation at $v$ is $0=0$. Therefore any choice of natural numbers for pseudo-sources and isolated vertices (with the remaining vertices of the reduced digraph sent to 0) yields a homomorphism from $S(\Gamma)$ to $\mathbb{N}$. This gives an isomorphism  from the free commutative monoid on the isolated vertices and the pseudo-sources of $\Gamma$ to the monoid of homomorphisms from $S(\Gamma)$ to $\mathbb{N}$ under addition.
\end{proof} \\

\begin{theorem} \label{27}
Let $\Gamma$ be a finite digraph. A dimension function is obtained by assigning arbitrary natural numbers to the maximal sinks and the maximal cycles of $\Gamma$. Hence the monoid of homomorphisms from $S(\Gamma) \cong \mathcal{V} (L(\Gamma))$ to $\mathbb{N}$ is isomorphic to the free commutative monoid generated by the maximal sinks and the maximal cycles of $\Gamma$.  
\end{theorem}

\begin{proof}
If $\Gamma'$ is a complete reduction of $\Gamma$ then there are one-to-one correspondences between the isolated vertices (respectively, the pseudo-sources) of $\Gamma'$ and the maximal sinks (respectively, the maximal cycles) of $\Gamma$: Under reduction the maximal sinks or cycles remain maximal because the new cycles created (if any) do not  involve vertices from which a maximal sink or cycle can be reached. 
Now Lemma \ref{psource} gives the desired conclusion.
\end{proof}\\

In particular, although a finite digraph may not have a unique complete reduction, all such must have the same number of isolated vertices and pseudo-sources.\\


Theorem \ref{27} which classifies all dimension functions of a finite digraph $\Gamma$ is essentially a refinement of  \cite[Theorem 6.4]{ko1} which determines the existence of a nonzero dimension function. 
Any dimension function of $\Gamma$ comes from a finite dimensional representation of $L(\Gamma)$ 
by \cite[Corollary 3.7]{ko1} but a representation may not be uniquely determined  by its dimension function. Theorem \ref{uzun1} in the next section will further refine Theorem \ref{27} by classifying all finite dimensional representations of $L(\Gamma)$ for a row-finite digraph $\Gamma$.

\section{The Classification of The Finite Dimensional Representations of $L(\Gamma)$ }

Let's first show that the reduction algorithm applied to a (row-finite) digraph gives a Morita equivalence at the level of Leavitt path algebras.  

\begin{theorem} \label{Morita} If $\Gamma'$ is a reduction of $\Gamma$ then $L(\Gamma)$ and $L(\Gamma')$ are Morita equivalent, that is, their $($unital$)$ module categories are equivalent. This equivalence preserves the subcategories of finite dimensional modules and modules of finite type. 
\end{theorem}

\begin{proof}
We may assume that $\Gamma'=(V',E', s',t')$ is a one step reduction of $\Gamma=(V,E,s,t)$ with $V'=V\setminus \{v\}$ as above. An $L(\Gamma)$-module $M$ is equivalent to a quiver representation $\rho$ of $\Gamma$ satisfying ($I$) by Theorem \ref{teorem}. We will construct the corresponding $L(\Gamma')$-module $M'$ as a quiver representation $\rho'$ of $\Gamma'$ as follows: $\rho'(w):=\rho(w)$ if $w \in V' \>$ and  $\> \rho'(e):=\rho(e)$ if $e \in E\cap E'$. For a new arrow $fg$ we define $\rho'(fg)$ to be the composition $\rho'(sf)  \stackrel{\rho(f)}{\longrightarrow} \rho(v) \stackrel{\rho(g)}{\longrightarrow} \rho'(tg) $. To see that $\rho'$ satisfies $(I)$ at $w \in V'$ we need to combine the isomorphisms $\rho(v) \longrightarrow \bigoplus_{sg=v } \rho(tg)$ and $\rho' (w) =\rho(w) \longrightarrow \bigoplus_{sf=w} \rho(tf)$ given by $(I)$ for $\rho$ at $v$ and $w$ (using that $\bigoplus_{sf=w} \rho(tf)$ can be broken up into summands with $tf=v$ and those with $tf \neq v$). The definition of $\rho(fg)$ yields the desired isomorphism at $w$. If $\phi$ is a morphism between quiver representations of $\Gamma$ satisfying (I) then the morphism $\phi'$ between the corresponding quiver representations of $\Gamma'$ is the restriction of $\phi$ to $\phi'$.\\

Given a quiver representation $\sigma$ of $\Gamma'$ we define the quiver representation $\tilde{\sigma}$ of $\Gamma$ as $\tilde{\sigma}(w):=\sigma(w)$ if $w \in V'$ and $\tilde{\sigma}(v):=\bigoplus_{sg=v} \sigma(tg)$ (note that $tg $ is in $V'$ since $v$ is a loopless non-sink); also $\tilde{\sigma}(e):=\sigma (e)$ if $e \in E \cap E'$, if $tf=v$ then  $\tilde{\sigma}(f):\tilde{\sigma}(sf) \stackrel{(\sigma (fg))_{sg=v}}{\longrightarrow } \tilde{\sigma}(v) = \bigoplus_{sg=v} \sigma(tg)$  and if $sh=v$ then $\tilde{\sigma}(h)$ is the projection from $\tilde{\sigma}(v)=\bigoplus_{sg=v} \sigma (tg)$ to the summand $\sigma (th)$. The condition $(I)$ at $v$ is immediate from the definitions. The definitions of $\tilde{\sigma}(v)$ and $\tilde{\sigma}(f)$  where $sf=w$ and $tf=v$ above yield $(I)$ for $\tilde{\sigma}$ at $w\neq v$ (using $t'fg=tg$). If $\psi$ is a morphism between quiver representations of $\Gamma'$ then the morphism $\tilde{\psi}$ between the corresponding quiver representations of $\Gamma$ is given by $\tilde{\psi} (w):= \psi(w)$ for $w\neq v$ and $\tilde{\psi }(v):=\bigoplus_{sg=v} \psi (tg)$.\\

By construction $(\tilde{\sigma})' =\sigma$ and $(\tilde{\psi})'=\psi$.  An isomorphism $\theta: \rho \longrightarrow \tilde{(\rho')}$ is given by $\theta_v=(\rho(te))_{se=v }$ and  $\theta_w=id_{\rho(w)}$ for $w \neq v$. Since $\theta$ is compatible with  $\tilde{(\phi')}$, the unital module categories of $L(\Gamma)$ and $L(\Gamma')$ are equivalent. Both constructions $\> \> \tilde{} \> \> $ and $ \> ' \>$ preserve finite support and finite dimensionality at every vertex, hence the subcategories of finite dimensional modules and modules of finite type are preserved. 
\end{proof}\\

\begin{remark}
An easy example of a Morita equivalence which does not preserve the subcategory of finite dimensional modules is given by the Leavitt path algebras consisting of a single vertex (which is isomorphic to $\mathbb{F}$) and the infinite line graph with one sink 

$$\Gamma : \qquad \xymatrix{& \cdots \> \> \ar [r] &{\bullet} \ar [r] & {\bullet} \ar [r] & {\bullet} & }$$
(which is isomorphic to $M_{\infty}(\mathbb{F})$, infinite matrices with finitely many nonzero entries.) The Morita equivalence is given by tensoring with the $(\mathbb{F}, M_{\infty}(\mathbb{F}))$-bimodule $\mathbb{F}^{(\mathbb{N})}$, finite $\mathbb{F}$-sequences. The latter has no finite dimensional modules (by Theorem 6.4 of \cite{ko1} this also follows from Theorem \ref{uzun1} below) since $\Gamma$ has no cycles and one maximal sink but with infinitely many predecessors.
\end{remark}

The proof of Theorem \ref{Morita} above showing that the module categories of $L:=L(\Gamma)$ and $L':= L(\Gamma')$ are equivalent is constructive and explicit at the level of quiver representations. The quiver representation for the reduction $\Gamma'$ is obtained by restricting the quiver representation of $\Gamma$ to the vertices of $\Gamma'$. In the opposite direction the representation at the omitted vertex $v$ is recovered by taking the direct sum of the representations at $\{te \vert se=v \}$. 
 A more typical Morita equivalence would be given by an $(L,L')$-bimodule $P$ (a progenerator for $\mathfrak{M}_{L'}$) and $(L,L')$-bimodule $Q$  (a progenerator for $\mathfrak{M}_{L}$) such that $ \underline{\> \> \> \> }\otimes_L P$ and $\underline{\>\> \> \> } \otimes_{L'} Q$ give the equivalences between the categories $\mathfrak{M}_{L'}$ and $\mathfrak{M}_{L}$ . Identifying $L'$ with the subalgebra $\{ \alpha \in L \mid v \alpha=0=\alpha v \} $ as above we have $P= \{ \alpha \in L \mid  \alpha v=0 \} = L' \oplus vL'$  and $Q= \{ \alpha \in L \mid  v\alpha =0 \} = \oplus_{_{v\neq w \in V}} wL $ .   \\

Next we want to classify finite dimensional modules of the Leavitt path algebra $L(\Gamma)$ of a row-finite digraph $\Gamma$ by determining all indecomposable finite dimensional modules. First we reduce the problem to a finite digraph $\Lambda$ (the support subgraph of the module). Then we apply the reduction algorithm to $\Lambda$ to obtain a disjoint union of loops and isolated vertices. Now the corresponding module is the direct sum of submodules associated to an isolated vertex or a loop and the classification of these is standard linear algebra. The result can then be pulled back to $L(\Gamma)$ via the explicit Morita equivalence of Theorem \ref{Morita}.

\begin{theorem} \label{uzun1}
A finite dimensional $L_{\mathbb{F}}(\Gamma)$-module where $\Gamma$ is a row-finite digraph has a functorial direct sum decomposition with the support subgraph of each summand being the predecessors of a maximal sink or a maximal cycle with finitely many predecessors. A summand $M$ corresponding to a sink $w$ is uniquely and completely determined by $dim_{\mathbb{F}} (Mw)$ (up to isomorphism). This summand is indecomposable if and only if $dim_{\mathbb{F}}(Mw)=1$ if and only if it is simple. A summand corresponding to a cycle decomposes into a direct sum of primary submodules, each associated with an irreducible polynomial $f(x)$ in $\mathbb{F}[x]$ with $f(0)=1$. The indecomposable summands of a primary submodule are uniquely and completely determined (up to isomorphism) by a positive integer, this integer is 1 if and only if the indecomposable is simple. 

\end{theorem}
\begin{proof}
The support subgraph $\Lambda:=\Gamma_{_M}$ of a finite dimensional $L(\Gamma)$-module $M$ is finite and its cycles have no exits by Lemma \ref{exits}. The $L(\Gamma)$-module structure is the restriction of the $ L(\Gamma)/I_{_M}\cong L(\Lambda)$-module structure.  A complete reduction  $\bar{\Lambda}$ of $\Lambda$ is a disjoint union of loops (in one-to-one correspondence with the cycles of $\Lambda$) and isolated vertices (in one-to-one correspondence with the sinks of $\Lambda$). 
We will work with the corresponding  $L(\bar{\Lambda})$-module via the Morita equivalence of Theorem \ref{Morita}.\\

$L(\bar{\Lambda})$ is isomorphic to the direct sum of its subalgebras given by the connected components of $\bar{\Lambda}$ which are loops or isolated vertices. Each subalgebra is of the form $L(\tilde{\Lambda})v$ and isomorphic to $\mathbb{F}$ if $v$ is a sink or $\mathbb{F}[x,x^{-1}]$ otherwise (where the loop $e_v$ at $v$ corresponds to $x$ and $e_v^*$ corresponds to $x^{-1}$). Similarly, any $L(\tilde{\Lambda})$-module N has a direct sum of decomposition $N=\oplus Nv$ where the subspaces $Nv$ are submodules in this case.
The $L(\bar{\Lambda})= \oplus L(\bar{\Lambda})v$ module structure of $N=\oplus Nv$ is given by coordinate-wise multiplication (i.e., $Nv$ is an $L(\bar{\Lambda})$-module and all other factors $L(\bar{\Lambda})w$ with $w \neq v$ annihilate $Nv$).\\

The support subgraph $\bar{\Lambda}_{Nv}$ of the $L(\bar{\Lambda})$-submodule $Nv$ is just the isolated vertex $v$ if $v$ is a sink and it consists of the vertex $v$ and the loop $e_v$ at $v$ when $v$ is not a sink. Also $L(\bar{\Lambda}_{Nv})\cong L(\bar{\Lambda})v$ which is (isomorphic to) $\mathbb{F}$ or $\mathbb{F}[x,x^{-1}]$ depending on whether $v$ is a sink or not. When $v$ is a sink, $Nv$ is a finite dimensional vector space determined (up to isomorphism) by its dimension. \\


 When $v$ is not a sink, $Nv$ is a finite dimensional $\mathbb{F}[x,x^{-1}]$-module 
which is a direct sum of its primary components each corresponding to an irreducible polynomial $f(x)$ in $\mathbb{F}[x]$ with $f(0)=1$. (These polynomials are the irreducible divisors of the characteristic polynomial of $e_v$ considered as an endomorphism of the vector space $Nv$.) The primary component corresponding to $f(x)$ is a direct sum of indecomposables, each specified by a positive integer $n$ (where $f(x)$ is the minimal polynomial of the linear transformation given by $e_v$ restricted to this primary submodule and $f(x)^n$ is its characteristic polynomial). The indecomposable is simple if and only if $n=1$.\\

The canonical decomposition $N=\oplus Nv$ of the $L(\bar{\Lambda})$-module gives a canonical decomposition of the corresponding $L(\Lambda)$-module via the explicit functorial Morita equivalence of Theorem \ref{Morita}. This decomposition can further be pulled back, via the epimorphism $L(\Gamma) \longrightarrow L(\Lambda)$, to a canonical decomposition of the original $L(\Gamma)$-module. When $v$ is a sink in $\bar{\Lambda}$, the support of the $L(\Gamma)$-module $M$ corresponding to the $L(\bar{\Lambda})$-module $Nv$ is the predecessors of the maximal sink $v$ of $\Gamma$. For each vertex $u$ of $\Gamma$, $Nu$ is isomorphic to a direct sum of copies of the vector space $Nv$ indexed by the paths from $u$ to $v$. For any arrow $e$ of $\Gamma$ each path from $te$ to $v$ yields a path from $se$ to $v$ starting with $e$. The linear transformation (projection) given by $e$ from $Mse$ to $Mte$ is identity on the summands corresponding to paths related as above and it is 0 otherwise.\\

If $v$ is not a sink in $\bar{\Lambda}$, so there is the loop $e_v$ based at $v$, then the support of the $L(\Gamma)$-module $M$ corresponding to $Nv$ is the predecessors of the maximal cycle $C_v$ in $\Gamma$ corresponding to the loop $e_v$ in $\bar{\Lambda}$. The vertex $v$ of $\bar{\Lambda}$ is also a vertex in $\Lambda$ (and hence $\Gamma$ since $\Lambda$ is a subgraph of $\Gamma$): it is
 the only vertex of the cycle $C_v$ which is not eliminated during the complete reduction of $\Lambda$ to $\bar{\Lambda}$. Let $\tilde{e}_v$ be the arrow on the cycle $C_v$ starting at $v$. Now $Mw \cong Nv$ for each vertex $w$ on $C_v$ and $Mu$ is (isomorphic to) the direct sum of copies of $Nv$ indexed by the paths from $u$ to $C_v$ (equivalently, the paths from $u$ to $v$ not containing $C_v$). The linear transformation corresponding to $\tilde{e}_v$ is given by $e_v$ while all other arrows on $C_v$ give $id_{Nv}$, the identity transformation. For any arrow $e$ of $\Gamma$, different from $\tilde{e}_v$ the linear transformation given by $e$ from $Mse$ to $Mte$ is the projection defined just as in the paragraph above. \\

The decompositions of the $L(\bar{\Lambda})$-module $Nv$ into primary summands and indecomposables carry over functorially to the $L(\Gamma)$-module $M$ and the simple $L(\bar{\Lambda})$-modules correspond to simple $L(\Gamma)$-modules.
\end{proof}\\

Consequently,  there are two kinds of finite dimensional indecomposable unital $L(\Gamma)$-modules when $\Gamma$ is row finite: if $M$ is of the first kind then $M$ is completely determined by a maximal sink $v$ with finitely many predecessors. The subspace $Mw$ for any vertex $w$ has dimension equal to the number of paths from $w$ to $v$ (hence the support of $M$ is the predecessors of $v$). The linear transformations given by the arrows are projections. These indecomposables are simple. \\

An indecomposable $M$ of the second kind is determined by a maximal cycle $C$ with finitely many predecessors, an irreducible $f(x) \in \mathbb{F}[x]$ with $f(0)=1$ and a positive integer $n$. Now, $Mv\cong \mathbb{F}[x]/ f(x)^n$ for any vertex $v$ on $C$. The linear transformation given by one of the arrows on $C$ corresponds to multiplication by $x$, the remaining arrows on $C$ give the identity transformation (isomorphism type of $M$ is independent of which arrow corresponds to multiplication by $x$). For any vertex $w$, $Mw$ is isomorphic to a direct sum of copies of $Mv$ indexed by the paths from $v$ to $C$ (which do not traverse $C$). The arrows that are not on the cycle $C$ give projections. These indecomposables are simple if and only if $n=1$.\\

 If $M$ and $N$ are simple then any nontrivial extension of $M$ by $N$ is indecomposable. Our classification shows that the only indecomposables of length 2 correspond to pairs $(C,f(x)^2)$, hence:

 \begin{corollary}
 If $M$ and $N$ are finite dimensional simple $L(\Gamma)$-modules then $Ext(M,N)=0$ unless $N\cong M$ and $M$ corresponds to a pair $(C,f(x))$ as above. 
  \end{corollary}
  
 The corollary above does not hold when we don't assume that $M$ and $N$ are finite dimensional. There are indecomposable $L(\Gamma)$-modules of arbitrary length with several non-isomorphic factors \cite{ko2}.\\

Any finite dimensional $L(\Gamma)$-module $M$ is, up to isomorphism, a unique (up to ordering) direct sum of the indecomposables described above. Thus Theorem 6.4 of \cite{ko1} stating that $L(\Gamma)$ has a nonzero finite dimensional module if and only if $\Gamma$ has a maximal sink or cycle with finitely many predecessors is now a consequence of Theorem \ref{uzun1}. While the existence of a nonzero finite dimensional representation of $L(\Gamma)$ depends only on the digraph $\Gamma$, 
the classification of all finite dimensional representations depends also on the coefficient field.  \\

When the field $\mathbb{F}$ is algebraically closed the simple modules obtained from a maximal cycle (with finitely many predecessors) are parametrized by $\mathbb{F}^\times$, the nonzero elements of the field. The indecomposables require an additional parameter, a positive integer. In particular, the category of finite dimensional representations of the Leavitt path algebra $L(\Gamma)$ of a finite digraph $\Gamma$ is tame in contrast to the finite dimensional quiver representations of $\Gamma$ which are wild unless the connected components of the undirected $\Gamma$ are Dynkin diagrams of types $A, \> D, \> E, \> \tilde{A}, \> \tilde{D}, \> \tilde{E}$ by a theorem of Gabriel \cite{dw05}. The algebra $L(\Gamma)$ has finite representation type, that is, there are only finitely many finite dimensional indecomposables (up to isomorphism) if and only if $\Gamma$ has no maximal cycles with finitely many predecessors and finitely many maximal sinks with finitely many predecessors.\\

 We can almost remove the qualification "with finitely many predecessors" if we consider modules of finite type instead of finite dimensional modules. Each maximal sink or cycle with infinitely many predecessors contributes $L(\Gamma)$-module of finite type. The converse, however, does not hold. There may be no sinks or cycles in $\Gamma$ but $L(\Gamma)$ may still have a module of finite type. For example
   $$\Gamma : \xymatrix{& {\bullet}^{v_1} \ar [r] ^{e_1} & {\bullet}^{v_2} \ar [r] ^{e_2} & {\bullet}^{v_3} \> \> \cdots & }$$
  The module $M$ of finite type is given by $Mv_i:=\mathbb{F}$, $i=1,2, \cdots$  and all arrows being $id_{\mathbb{F}}$.\\ 

 \begin{example}
 The two digraphs below have  no maximal sinks or maximal cycles, hence their Leavitt path algebras have no nonzero finite dimensional representations.
 
 $$   
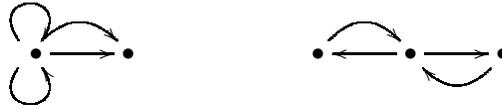
\xymatrix{  &\> \> { \bullet}
\ar@(ul,ur) \ar@(dl,dr)  \ar@/^1pc/[r] \ar[r]  & {\bullet} }  \qquad \quad   \xymatrix{& {\bullet} \ar@/^1pc/[r] &\ar [l]  {\bullet}  \ar [r]  & {\bullet} \ar@/^1pc/[l] }$$\\

\end{example}

\begin{example}
Finite dimensional representations of  the Leavitt path algebras of the five digraphs below are all equivalent since they have no maximal sinks and the predecessor subgraphs of their unique maximal cycle are isomorphic. \\


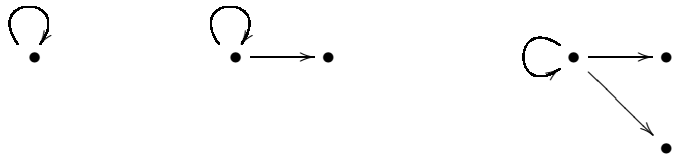
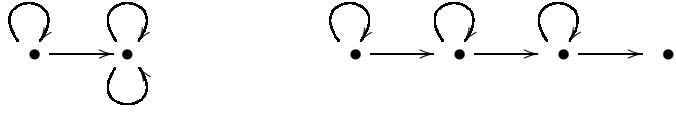
$$  \xymatrix{  &\> \> { \bullet}
\ar@(ul,ur) }  \qquad   \quad \xymatrix{  &\> \> { \bullet} \ar[r]
\ar@(ul,ur)   &{ \bullet} } \qquad  \qquad \quad \quad \qquad \xymatrix{ {\bullet} \ar@(ul,dl) \ar@{->}[r] \ar@{->}[dr] &\bullet \\
                 & \bullet  }$$

$$ \xymatrix{  &\> \> { \bullet} \ar[r]
\ar@(ul,ur)   &{ \bullet} 
\ar@(ul,ur) \ar@(dl,dr) } \quad \qquad \quad \xymatrix{  &\> \> { \bullet} \ar[r]
\ar@(ul,ur)   &\> \> { \bullet} \ar[r]
\ar@(ul,ur)   & \>\>  { \bullet}  \ar[r] \ar@(ul,ur)   &\> \> { \bullet} } 
 $$\\\\


\end{example}

A somewhat different viewpoint is to start with the finite dimensional $L(\Gamma)$-module $M$ and to determine all its invariants. We arbitrarily fix a base vertex $w$ for each maximal cycle with finitely many predecessors and let $W$ be the set of the base vertices of maximal cycles and the maximal sinks with finitely many predecessors. Regarding $M$ as an $L(\Gamma_{_M})$-module (where $\Gamma_{_M}=(V_{_M}, E_{_M})$ is the support subgraph of $M$), we get the corresponding $L(\bar{\Gamma}_{_M})$-module $N$ using the explicit Morita equivalence of Theorem \ref{Morita}. Here $\bar{\Gamma}_{_M}$ is the complete reduction of $\Gamma_{_M}$ in which all vertices not in $W$ are eliminated. The vertex set of $\bar{\Gamma}_{_M}$ is $V_{_M} \cap W$
since $\bar{\Gamma}_{_M}$ is a disjoint union of loops and isolated vertices as explained above.\\

In the decomposition $N=\oplus_{w \in W} Nw$ each $Nw$ is an $L(\bar{\Gamma}_{_M})$-submodule. Let $M_{\leadsto w}$ be the $L(\Gamma_{_M})$-module corresponding to the $L(\bar{\Gamma}_{_M})$-module $Nw$ via Theorem \ref{Morita}. Regarding $M_{\leadsto w}$ as an $L(\Gamma)$-module via the epimorphism $L(\Gamma) \longrightarrow L(\Gamma_{_M})$ we have $M=\oplus_{w \in W} M_{\leadsto w} \> $; this is the functorial decomposition of Theorem \ref{uzun1}. The support of $M_{\leadsto w}$ is $\Gamma_{\leadsto w} \> $ (note that these support subgraphs may not be disjoint for distinct $w$).  \\

 The submodule $M_{\leadsto w}$ is determined by the vector space $Mw$ if $w$ is a sink, it is determined by the vector space $Mw$ and the  invertible linear operator on $Mw$ given by right multiplication with $C_w$, the cycle containing $w$ (considered as an element of $L(\Gamma)$). 
 If  $f:M \longrightarrow M'$ is an $L(\Gamma)$-module homomorphism where $M$ and $M'$ are both finite dimensional then we may consider $f$ an $L(\Lambda)$-morphism where $\Lambda=\Gamma_{_M} \cup \Gamma_{_{M'}} \>$. The complete reduction $\bar{\Lambda}$ of $\Lambda$ with vertex set $W \cap (V_{_M} \cup V_{_{M'}})$ is again a disjoint union of loops and isolated vertices. Since $\bar{f} (Nw) \subseteq N'w$ for the corresponding $L(\bar{\Lambda})$-morphism $\bar{f} : N \longrightarrow N'$ we have $f(M_{\leadsto w}) \subseteq M_{\leadsto w}' \>$.\\

The $L(\Gamma)$-morphism $f \mid_{_{M_{\leadsto w}}} : M_{\leadsto w} \longrightarrow {M'}_{\leadsto w}$ is determined by the linear transformation $f \mid_{_{Mw}} :Mw \longrightarrow M'w $ completely. When $w$ is not a sink the linear transformation $f \mid_{_{Mw}} $ intertwines with $C_w$ (equivalently, $f \mid_{_{Mw}}$ is an $\mathbb{F}[x,x^{-1}]$-morphism where $x$ acts as $C_w$). Conversely, for any finite subset of $W$, given arbitrary finite dimensional vector spaces $Mw$, invertible operators $C_w: Mw \longrightarrow Mw $ when $w$ is not a sink we can construct an $L(\bar{\Lambda})$-module $N$ with $Nw :=Mw$ and the linear transformation corresponding to $e_w$ (the loop based at w) is $C_w \>$. The corresponding $L(\Gamma)$-module $M$ (via Theorem \ref{Morita} and the epimorphism $L(\Gamma) \longrightarrow L(\Gamma_{_M})$) has the prescribed $Mw$ and $C_w$ for each $w \in W$. Similarly, an $L(\Gamma)$-morphism $f :M \longrightarrow M'$ with arbitrarily specified $f \mid_{_{Mw}}$ satisfying the intertwining conditions can also be realized.\\

Paraphrasing the discussion above: The category of finite dimensional $L(\Gamma)$-modules for a row-finite digraph $\Gamma$ is 
equivalent to the direct sum of the categories indexed by $W$ where each summand corresponding to a sink is the category of finite dimensional vector spaces and each remaining summand is the category of finite dimensional $\mathbb{F}[x,x^{-1}]$-modules. That is, the category of finite dimensional $L(\Gamma)$-modules for a row-finite digraph $\Gamma$ is equivalent to the category of finite dimensional $L(\Lambda)$-modules where $\Lambda$ is a disjoint union of loops (corresponding to the maximal cycles of $\Gamma$ with finitely many predecessors) and isolated vertices (corresponding to the maximal sinks of $\Gamma$ with finitely many predecessors):

\begin{theorem} \label{notasyon}
If $\Gamma$ is a row-finite digraph then $$\mathfrak{M}_L^{^{fd}} \backsimeq \Big( \bigoplus_{_{\mathcal{S}} }
\mathfrak{M}_{\mathbb{F}}^{^{fd}} \Big) \oplus \Big(\bigoplus_{_{\mathcal{T}} } \mathfrak{M}_{\mathbb{F} [x,x^{-1}]}^{^{fd}}\Big)$$
where $L:=L_{\mathbb{F}}(\Gamma)$ ,  $\mathfrak{M}_A^{^{fd}}$ is the category of finite dimensional $A$-modules, $S$ is the set of maximal sinks with finitely many predecessors in $\Gamma$, $\mathcal{T}$ is the set of maximal cycles with finitely many predecessors in $\Gamma$ and $\backsimeq$ denotes equivalence of categories. 
\end{theorem}

We can get a more module theoretic description of the classification of the indecomposable finite dimensional $L(\Gamma)$-modules by unraveling the correspondence between quiver representations and $L(\Gamma)$-modules of Theorem \ref{teorem}:
Given a finite dimensional indecomposable $L(\Gamma)$-module its support $\Gamma_M$ contains either a unique sink $w$ or a unique cycle $C$. If $\Gamma_M$ has a sink $w$ then $M\cong wL(\Gamma)$, a finite dimensional projective $L(\Gamma)$-module which is also simple. These are the only projective indecomposable (or simple) finite dimensional $L(\Gamma)$-modules. 
Any finite dimensional projective $L(\Gamma)$-module is isomorphic to a direct sum of these $\{wL(\Gamma)\}$ where $w$ is a maximal sink with finitely many predecessors.\\

If $\Gamma_M$ contains a cycle $C$ then $C$ is the unique cycle in $\Gamma_M=\Gamma_{\leadsto C}$ and in $\Gamma_M$ $C$ has no exit. $M$ corresponds to a pair $(C,f(x))$ in the classification, we can recover the polynomial $f(x)$ as follows: if $v=sC$ and $\varphi$ is the linear operator on the vector space $Mv$ given by right multiplication with $C$ then $f(x)$ is $det(1-x\varphi) $, essentially the characteristic polynomial of $\varphi$.\\

  Conversely, given $(C,f(x)^n)$ where $C$ is a maximal cycle with finitely many predecessors and $f(x) \in \mathbb{F}[x]$ an irreducible polynomial with constant term 1, the corresponding finite dimensional indecomposable $M$ is $\mathbb{F}[x,x^{-1}]/ (f(x)^n) \otimes vL(\Gamma_{\leadsto C})$ where $v=sC$ and tensor product is over $\mathbb{F}[x,x^{-1}] \cong vL(\Gamma_{\leadsto C})v$ (the isomorphism is by Lemma \ref{ez}). Note that $vL(\Gamma_{\leadsto C})$ is a $(vL(\Gamma_{\leadsto C})v, L(\Gamma_{\leadsto C}))$-bimodule, so $M$ is an $L(\Gamma_{\leadsto C})$-module and hence inherits an $L(\Gamma)$-module structure since $ L(\Gamma_{\leadsto C}) \cong L(\Gamma)/ (V\setminus V_{\leadsto C})$.\\
  
  The dimension function $d(u):=dim^\mathbb{F}(Mu)$ can also be computed in terms of $(C,f(x))$: if $v=sC$ and $P_v^u$  denotes the set of paths $p$ such that $sp=u$, $tp=v$ and $p$ does not contain $C$ then $d(u)= \vert P_v^u\vert degf(x)$. Similarly, when $M$ corresponds to maximal sink $w$ with finitely many predecessors, i.e., $M\cong wL(\Gamma)$ then $d(u)=\vert P_w^u\vert$  where $P_w^u$ is the set of paths from $u$ to $w$.\\
  
  Another description of the indecomposable module $M$ corresponding to the pair $(C, f(x))$ is $vL(\Gamma)/f(C^*)L(\Gamma)$ where $v=sC$ \cite{ko2}. A consequence is that all finite dimensional $L(\Gamma)$-modules are finitely presented. 
  Also $M$ is a rational Chen module if and only if $f(x)=1-x$ and  $M$ is a twisted rational Chen module if and only if $f(x)=1-\lambda x$ where $0\neq \lambda \in \mathbb{F}$ \cite{che15}, \cite{ko2}. If $f(x)$ is an irreducible polynomial (as always $f(0)=1$) then the corresponding module $M$ is simple but not a Chen module. \\
  
  \textbf{Acknowledgement}\\

Both authors were partially supported by TUBITAK grant 115F511.




$^*$ Department of Mathematics and Computer Science, \\ \.{I}stanbul K\"{u}lt\"{u}r University, Atak\"{o}y, \.{I}stanbul, TURK\.{I}YE\\ 
E-mail: akoc@iku.edu.tr\\

$^{**}$ Department of Mathematics, \\ University of Oklahoma, Norman, OK, USA \\
E-mail: mozaydin@ou.edu

\end{document}